\documentclass{article}

\usepackage{stmaryrd,amssymb}
\newcommand{\llb}{\llbracket}
\newcommand{\rrb}{\rrbracket}
\newcommand{\la}{\langle}
\newcommand{\ra}{\rangle}
\newcommand{\cat}{{^\smallfrown}}
\newcommand{\res}{{\upharpoonright}}
\newcommand{\NN}{\mathbb{N}}

\newcommand{\NNNN}{\NN^\NN}
\newcommand{\oi}{\{0,1\}}
\newcommand{\oist}{\oi^\ast}
\newcommand{\oiNN}{\oi^\NN}
\newcommand{\Pioi}{\Pi^0_1}

\newcommand{\leT}{\leq_{\mathrm{T}}}
\newcommand{\nleT}{\nleq_{\mathrm{T}}}
\newcommand{\lT}{<_{\mathrm{T}}}
\newcommand{\geT}{\geq_{\mathrm{T}}}
\newcommand{\gT}{>_{\mathrm{T}}}
\newcommand{\eqT}{\equiv_{\mathrm{T}}}
\newcommand{\degT}{\mathrm{deg}_{\mathrm{T}}}
\newcommand{\lett}{\leq_{\mathrm{tt}}}

\newcommand{\eqtt}{\equiv_{\mathrm{tt}}}
\newcommand{\degtt}{\mathrm{deg}_{\mathrm{tt}}}

\newcommand{\Ew}{\mathcal{E}_{\mathrm{w}}}
\newcommand{\CPA}{\mathrm{CPA}}
\newcommand{\PA}{\mathsf{PA}}
\newcommand{\MLR}{\mathrm{MLR}}
\newcommand{\LR}{\mathrm{LR}}

\newcommand{\liff}{\Leftrightarrow}
\newcommand{\Sent}{\textrm{Sent}}

\newcommand{\Comp}{\textrm{Comp}}
\newcommand{\aaa}{\mathbf{a}}
\newcommand{\bbb}{\mathbf{b}}

\usepackage{amsthm}
\theoremstyle{definition}
\newtheorem{thm}{Theorem}[section]

\newtheorem{lem}[thm]{Lemma}
\newtheorem{dfn}[thm]{Definition}
\newtheorem{rem}[thm]{Remark}

\usepackage[backref=page,colorlinks=true,allcolors=blue]{hyperref}

\title{Pseudojump inversion in special r.\ b.\ $\Pi^0_1$ classes}

\author{Hayden R. Jananthan\\
  Department of Mathematics\\
  Vanderbilt University\\
  Nashville, TN 37203, USA\\
  \href{https://my.vanderbilt.edu/haydenjananthan}
  {https://my.vanderbilt.edu/haydenjananthan}\\
  \href{mailto:hayden.r.jananthan@vanderbilt.edu}
  {hayden.r.jananthan@vanderbilt.edu}\\
  {\ }\\
  Stephen G. Simpson\\
  Department of Mathematics\\
  Vanderbilt University\\
  Nashville, TN 37203, USA\\
  \href{http://www.math.psu.edu/simpson}
  {http://www.math.psu.edu/simpson}\\
  \href{mailto:sgslogic@gmail.com}{sgslogic@gmail.com}}

\date{First draft: June 28, 2019\\
  This draft: \today}

\begin{document}

\maketitle

\begin{abstract}
  The Jump Inversion Theorem says that for every real $A\geT0'$ there
  is a real $B$ such that $A\eqT B'\eqT B\oplus0'$.  A known
  refinement of this theorem says that we can choose $B$ to be a
  member of any special $\Pi^0_1$ subclass of $\oiNN$.  We now
  consider the possibility of analogous refinements of two other
  well-known theorems: the Join Theorem -- for all reals $A$ and $Z$
  such that $A\geT Z\oplus0'$ and $Z\gT0$, there is a real $B$ such
  that $A\eqT B'\eqT B\oplus0'\eqT B\oplus Z$ -- and the Pseudojump
  Inversion Theorem -- for all reals $A\geT0'$ and every $e\in\NN$,
  there is a real $B$ such that $A\eqT B\oplus W_e^B\eqT
  B\oplus0'$. We show that in these theorems, $B$ can be found in some
  special $\Pi^0_1$ subclasses of $\oiNN$ but not in others.
\end{abstract}

\newpage

\tableofcontents

\section{Introduction}

We begin with a well-known theorem.

\begin{thm}[Jump Inversion Theorem, due to Friedberg, see
  \protect{\cite[Theorem 13.3.IX]{rogers}}]\label{thm:jump}
  Given a real $A\geT0'$ we can find a real $B$ such that $A\eqT B'
  \eqT B \oplus 0'$.
\end{thm}

Here $0'$ denotes the Halting Problem, $B'$ denotes the Turing jump of
$B$, and $\eqT$ denotes Turing equivalence.  There are two
important related theorems, which read as follows.

\begin{thm}[Join Theorem, due to Posner and Robinson
  \protect{\cite[Theorem 1]{po-ro}}, see also
  \protect{\cite[Theorem 2.1]{jo-sh-1}}]\label{thm:join}
  Given reals $A$ and $Z$ such that $A\geT Z\oplus0'$ and $Z\gT0$, we
  can find a real $B$ such that $A\eqT B'\eqT B\oplus 0'\eqT B\oplus
  Z$.
\end{thm}

\begin{thm}[Pseudojump Inversion Theorem, due to Jockusch and Shore
  \protect{\cite[Theorem 2.1]{jo-sh-1}}] \label{thm:pjinv} Given a
  real $A\geT0'$ and an integer $e\in\NN$, we can find a real $B$ such
  that $A\eqT J_e(B)\eqT B\oplus 0'$.
\end{thm}
Here $J_e$ denotes the $e$th pseudojump operator, defined by
$J_e(B)=B\oplus W_e^B$ where $\la W_e^B\mid e\in\NN\ra$
is a fixed standard recursive enumeration of the $B$-recursively
enumerable subsets of $\NN$.

Our results in this paper concern special $\Pi^0_1$ classes.
Following \cite{js2} we define a \emph{special $\Pi^0_1$ class} to be
a nonempty $\Pi^0_1$ subset of $\NNNN$ with no recursive elements.
Here $\NNNN$ is the Baire space, but as in \cite{js2,js} we focus on
special $\Pi^0_1$ subclasses of $\oiNN$, the Cantor space.  On the
other hand, because of \cite[Theorem 4.10]{massrand} our results
concerning special $\Pi^0_1$ subclasses of $\oiNN$ will also apply to
\emph{recursively bounded} special $\Pi^0_1$ subclasses of $\NNNN$.
Following \cite{js} we use ``r.\ b.''\ as an abbreviation for
``recursively bounded.''

We now recall another known theorem, which says that we can find the
$B$ for Theorem \ref{thm:jump} in any special $\Pi^0_1$ subclass of
the Cantor space.
\begin{thm}[due to Jockusch and Soare \protect{\cite[just
    after the proof of Theorem 2.1]{js}}]\label{thm:jump-special}
  Let $P$ be a special $\Pi^0_1$ subclass of $\oiNN$.  Given a real
  $A\geT0'$, we can find a real $B\in P$ such that $A\eqT B'\eqT
  B\oplus 0'$.
\end{thm}

This refinement of Theorem \ref{thm:jump} suggests a question as to
the existence or nonexistence of analogous refinements of Theorems
\ref{thm:join} and \ref{thm:pjinv}.  In order to state our results
concerning this question, we make the following definition.
\begin{dfn}\label{dfn:props}
  Let $P\subseteq\NNNN$ be a subclass of the Baire space.
  \begin{enumerate}
  \item $P$ has the \emph{Join Property} if for all reals $A$ and $Z$
    such that $A\geT Z\oplus0'$ and $Z\gT0$, there exists
    $B\in P$ such that $A\eqT B'\eqT B\oplus0'\eqT
    B\oplus Z$.
  \item $P$ has the \emph{Pseudojump Inversion Property} if for all
    reals $A\geT0'$ and all $e\in\NN$, there exists $B\in
    P$ such that $A\eqT J_e(B)\eqT B\oplus0'$.
  \end{enumerate}
\end{dfn}

The question that we are asking is, which special $\Pi^0_1$ subclasses
of the Cantor space $\oiNN$ have one or both of the properties in
Definition \ref{dfn:props}?  The purpose of this paper is to present
some partial results in this direction, as follows.  Let $\CPA$ be the
special $\Pi^0_1$ subclass of $\oiNN$ consisting of all complete,
consistent extensions of Peano Arithmetic.  We show that $\CPA$ has
both the Join Property and the Pseudojump Inversion Property.  More
generally, we show that any special $\Pi^0_1$ class which is Turing
degree isomorphic to $\CPA$ has both of these properties.  For
example, this holds for the $\Pi^0_1$ class $\mathrm{CZF}$ consisting
of all complete, consistent extensions of Zermelo-Fraenkel Set Theory
(assuming that $\mathrm{CZF}$ is nonempty), and for the $\Pi^0_1$
class $\mathrm{DNR}_2$ consisting of all $\oi$-valued diagonally
nonrecursive functions.  On the other hand, let $P$ be a special
$\Pi^0_1$ subclass of $\oiNN$ which is of positive measure.  Citing a
theorem of Nies and an observation of Patrick Lutz, we note that $P$
need not have the Join Property, but $P$ and indeed any special
$\Pi^0_1$ class which is Turing degree isomorphic to $P$ has the
Pseudojump Inversion Property.  Finally, we construct a special
$\Pi^0_1$ subclass of $\oiNN$ which has neither the Join Property nor
the Pseudojump Inversion Property.  We do not know whether there
exists a special $\Pi^0_1$ subclass of $\oiNN$ which has the Join
Property but not the Pseudojump Inversion Property.

Here is an outline of the paper.  In \S2 we present the classical
proofs of Theorems \ref{thm:join} and \ref{thm:pjinv}.  In \S3 we use
the G\"odel-Rosser Incompleteness Theorem \cite{mendelson} to prove
that $\CPA$ has the Join Property and the Pseudojump Inversion
Property.  In \S4 we present some results about Turing degree
isomorphism of $\Pi^0_1$ classes, and we apply these results to show
that any $\Pi^0_1$ class which is Turing degree isomorphic to $\CPA$
has the Join Property and the Pseudojump Inversion Property.  In \S5
we show that special $\Pi^0_1$ classes consisting of Martin-L\"{o}f
random reals have the Pseudojump Inversion Property but not the Join
Property.  In \S6 we use a priority argument to construct a special
$\Pi^0_1$ subclass $Q\subseteq\oiNN$ which has neither the Join
Property nor the Pseudojump Inversion Property.

\begin{rem}
  The lattice $\Ew$ of all Muchnik degrees of nonempty $\Pi^0_1$
  subclasses of $\oiNN$ is of great interest; see for instance the
  references in \nocite{cie2015}\cite{dou-tutorial}.  Therefore, it is
  natural to ask about Muchnik degrees in $\Ew$ corresponding to the
  properties in Definition \ref{dfn:props}.  Our results in this paper
  shed some light on this question.  First, it is known \cite[Theorem
  6.8]{massrand} that any $\Pi^0_1$ subclass of $\oiNN$ which is
  Muchnik equivalent to $\CPA$ -- i.e., of Muchnik degree
  $\mathbf{1}$, where $\mathbf{1}$ denotes the top degree in $\Ew$ --
  is Turing degree isomorphic to $\CPA$.  Therefore, our results in
  \S4 below imply that all $\Pi^0_1$ classes of Muchnik degree
  $\mathbf{1}$ have both of the properties in Definition
  \ref{dfn:props}.  Second, it is easy to see that each degree in
  $\Ew$ contains a $\Pi^0_1$ subclass of $\oiNN$ which includes $\CPA$
  and hence again has both properties.  On the other hand, by \S5
  below we have a nonzero degree $\mathbf{r}\in\Ew$ containing a
  special $\Pi^0_1$ subclass of $\oiNN$ which has one property but not
  the other.  Moreover, in \S6 below we construct a nonzero degree in
  $\Ew$ containing a special $\Pi^0_1$ subclass of $\oiNN$ which does
  not have either property.
\end{rem}

In the remainder of this section, we fix some notation and
terminology.

We write $f:{\subseteq}A\to B$ to mean that $f$ is a \emph{partial
  function} from $A$ to $B$, i.e., a function with domain $\subseteq
A$ and range $\subseteq B$.  For $a\in A$ we write $f(a){\downarrow}$
and say that $f(a)$ \emph{converges} or $f(a)$ is \emph{defined}, if
$a\in$ the domain of $f$.  Otherwise we write $f(a){\uparrow}$ and say
that $f(a)$ \emph{diverges} or $f(a)$ is \emph{undefined}.  For $a\in
A$ and $f,g:{\subseteq}A\to B$, we write $f(a)\simeq g(a)$ to
mean that either $(f(a){\downarrow}$ and $g(a){\downarrow}$ and
$f(a)=g(a))$ or $(f(a){\uparrow}$ and $g(a){\uparrow})$.  We write
$f(a){\downarrow}=b$ to mean that $f(a){\downarrow}$ and $f(a)=b$.

The set of all natural numbers is $\NN=\{0,1,2,\ldots\}$.  For any set
$A$ let $A^\NN$ denote the set of all sequences $X:\NN\to A$.  Thus
$\NNNN$ and $\oiNN$ are the \emph{Baire space} and the \emph{Cantor
  space} respectively.  The points $X\in\NNNN$ or $X\in\oiNN$ are
sometimes called \emph{reals}.  We sometimes identify $\oiNN$ with the
powerset of $\NN$ in the usual manner.  For $X,Y\in A^\NN$ the
\emph{join} $X\oplus Y\in A^\NN$ is given by $(X\oplus Y)(2n)=X(n)$,
$(X\oplus Y)(2n+1)=Y(n)$, and for $P,Q\subseteq A^\NN$ we write
$P\times Q=\{X\oplus Y\mid X\in P$ and $Y\in Q\}$.  Note in particular
that $P\times Q\subseteq\NNNN$ whenever $P,Q\subseteq\NNNN$, and
$P\times Q\subseteq\oiNN$ whenever $P,Q\subseteq\oiNN$.

For any set $A$ let $A^\ast$ be the set of \emph{strings}, i.e.,
finite sequences, of elements of $A$.  For $a_0,\ldots,a_{m-1}\in A$
we let $\la a_0,\ldots,a_{m-1}\ra$ denote the string $\sigma\in
A^\ast$ of length $m$ given by $\sigma(i)=a_i$ for all $i<m$.  In this
case we write $|\sigma|=m$, and for each $k\le m$ we write $\sigma\res
k=\la a_0,\ldots,a_{k-1}\ra$.  We also write $A^m=\{\sigma\in
A^\ast\mid|\sigma|=m\}$ and $A^{\ge m}=\{\sigma\in
A^\ast\mid|\sigma|\ge m\}$ and $A^{<m}=\{\sigma\in
A^\ast\mid|\sigma|<m\}$.  For strings $\sigma=\la
a_0,\ldots,a_{m-1}\ra$ and $\tau=\la b_0,\ldots, b_{n-1}\ra$, their
\emph{concatenation} is $\sigma\cat\tau=\la a_0,\ldots,
a_{m-1},b_0,\ldots,b_{n-1}\ra$.  Note that
$|\sigma\cat\tau|=|\sigma|+|\tau|$.  We say that $\sigma$ is an
\emph{initial segment} of $\tau$, or equivalently $\tau$ is an
\emph{extension} of $\sigma$, written $\sigma\subseteq\tau$, if
$\tau\res|\sigma|=\sigma$, i.e., if $\tau=\sigma\cat\rho$ for some
string $\rho$.  We write $\sigma\subset\tau$ to mean that
$\sigma\subseteq\tau$ and $\sigma\ne\tau$.  We say that $\sigma\in
A^\ast$ is an \emph{initial segment} of $X\in A^\NN$, or equivalently
$X$ is an \emph{extension} of $\sigma$, written $\sigma\subset X$, if
$X\res|\sigma|=\sigma$, i.e., if $\sigma=\la
X(0),\ldots,X(|\sigma|-1)\ra$.

Fix a standard recursive enumeration $\la\varphi_e^{(k)}\mid
e\in\NN\ra$ of the $k$-place partial recursive functions
$\psi:{\subseteq}\NN^k\to\NN$.  Here $e$ is called an \emph{index} of
$\varphi_e^{(k)}$.  Likewise, for $X\in\NNNN$ we have a standard
recursive enumeration $\la\varphi_e^{(k),X}\mid e\in\NN\ra$ of the
$k$-place partial functions $\psi:{\subseteq}\NN^k\to\NN$ which are
partial recursive relative to $X$.  Here $e$ is again called an
\emph{index} of $\varphi_e^{(k),X}$, while $X$ is called an
\emph{oracle}.  We write
$W_e^X=\{i\in\NN\mid\varphi_e^{(1),X}(i){\downarrow}\}$ and this gives
a standard recursive enumeration of the subsets of $\NN$ which are
recursively enumerable relative to $X$.  We also write
$P_e=\{X\in\NNNN\mid\varphi_e^{(1),X}(0){\uparrow}\}$ and this gives a
standard recursive enumeration of the $\Pioi$ subclasses of $\NNNN$.

For $X\in\NNNN$ let $X'\in\NNNN$ denote the \emph{Turing jump} of $X$.
For definiteness we take $X'$ to be (the characteristic function of)
the set
\begin{center}
  $\{e\in\NN\mid0\in W_e^X\}=\{e\in\NN\mid X\notin P_e\}$.
\end{center}
For each $e\in\NN$ we have a \emph{pseudojump operator}
$J_e:\NNNN\to\NNNN$ given by $J_e(X)=X\oplus W_e^X$.  We write $\leT$
for \emph{Turing reducibility} and $\eqT$ for \emph{Turing
  equivalence}.  We write $\lett$ for \emph{truth-table
  reducibility}\footnote{See for instance \cite[\S\S8.3,9.6]{rogers}
  and \cite[\S\S4,5]{massrand}.} and $\eqtt$ for \emph{truth-table
  equivalence}.

\section{The Join Theorem and the Pseudojump Inversion Theorem}

In this section we present the classical proofs of the Join
Theorem and the Pseudojump Inversion Theorem.  Later we shall build on
these proofs in order to obtain our refinements.

We begin with the Join Theorem, Theorem \ref{thm:join}.

\begin{proof}[Proof of Theorem \ref{thm:join}]
  Let $A$ and $Z$ be reals such that $A\geT Z\oplus0'$ and $Z\gT0$.
  It will suffice to find a real $B$ such that $A\eqT B'\eqT
  B\oplus0'\eqT B\oplus Z$.  We shall construct a sequence of strings
  $\sigma_0\subset\sigma_1\subset
  \cdots\subset\sigma_j\subset\sigma_{j+1}\subset\cdots$ in
  $\NN^\ast$, and the desired real $B\in\NNNN$ will be obtained as the
  limit $\bigcup_j\sigma_j$ of this sequence.

  Let us regard $Z$ as a subset of $\NN$.  Since $Z$ is not
  recursive, at least one of $Z$ and its complement
  $\NN\setminus Z$ is not $\Sigma^0_1$.  And then, since $Z$ is
  Turing equivalent to $\NN\setminus Z$, we may safely assume
  that $\NN\setminus Z$ is not $\Sigma^0_1$.
  
  Stage $j=0$.  Let $\sigma_0 = \la\ra$.

  Stage $j=2n+1$.  By induction we have $\sigma_{2n}$.  Define
  \begin{center}
    $S_n = \{ i \in \NN \mid \exists \sigma\, (\sigma \supseteq
    \sigma_{2n} \cat \la\underbrace{0,\ldots,0}_i,1\ra$ and
    $\varphi_{n,|\sigma|}^{(1),\sigma}(0) \downarrow)\}$
  \end{center}
  which is $\Sigma^0_1$.  Since $\NN\setminus Z$ is not $\Sigma^0_1$,
  there are infinitely many $i$ such that $i\in Z\liff i\in
  S_n$.  Choose the first such $i$, and use this $i$ to define
  $\sigma_{2n+1}$ as follows.  If $i\in Z$, then $i\in S_n$ so let
  $\sigma_{2n+1}=$ the least
  $\sigma\supseteq\sigma_{2n}\cat\la\underbrace{0,\ldots,0}_i,1\ra$
  such that $\varphi_{n,|\sigma|}^{(1),\sigma}(0) \downarrow$.  If
  $i\notin Z$, let
  $\sigma_{2n+1}=\sigma_{2n}\cat\la\underbrace{0,\ldots,0}_i,1\ra$.

  Stage $j=2n+2$.  Let $\sigma_{2n+2}=\sigma_{2n}\cat\la A(n)\ra$.

  This completes the definition of $\sigma_j$ for all $j\in\NN$.  We
  now show that $B=\bigcup_j\sigma_j$ has the desired properties.  We
  begin with the following observations.
  \begin{itemize}
  \item If $\sigma_{2n}$ is known, then the $i$ in stage $2n+1$ can be
    found recursively in $Z\oplus 0'$ (because $S_n$ is computable
    from $0'$), or recursively in $B$ (because there is exactly one
    $i$ such that
    $\sigma_{2n}\cat\la\underbrace{0,\ldots,0}_i,1\ra\subset B$).  And
    then, using this $i$, $\sigma_{2n+1}$ can be found recursively in
    $Z$ (to test whether $i\in Z$), or recursively in $0'$ (to test
    whether $i\in S_n$).
  \item If $\sigma_{2n+1}$ is known, then $\sigma_{2n+2}$ can be found
    recursively in $A$ (by the definition of $\sigma_{2n+2}$), or
    recursively in $B$ (because
    $\sigma_{2n+2}=\sigma_{2n+1}\cat\la
    B(|\sigma_{2n+1}|\ra$), hence also recursively in $B\oplus Z$
    or in $B\oplus0'$.
  \end{itemize}
  Combining these observations, we see that the entire sequence
  $\la\sigma_j\mid j\in\NN\ra$ is $\leT B\oplus
  Z$, and $\leT B\oplus0'$, and $\leT A$ (using the
  hypothesis $Z\oplus0'\leT A$).  We also have:
  \begin{itemize}
  \item $B\oplus Z\leT A$, because $B=\bigcup_j\sigma_j\leT A$ and by
    hypothesis $Z\leT A$.
  \item $B'\leT\la\sigma_j\mid j\in\NN\ra$,
    because $n\in
    B'\liff\varphi_{n,|\sigma_{2n+1}|}^{(1),\sigma_{2n+1}}(0)
    \downarrow$.
  \item $A\leT\la\sigma_j\mid j\in\NN\ra$,
    because $A(n)=\sigma_{2n+2}(|\sigma_{2n+1}|)$ for all $n$.
  \end{itemize}
  Thus $A\leT\la\sigma_j\mid j\in\NN\ra\leT
  B\oplus Z\leT A$, and $B'\leT\la\sigma_j\mid
  j\in\NN\ra\leT B\oplus0'\leT B'$, hence
  $A\eqT B'\eqT B\oplus0'\eqT B\oplus Z$, Q.E.D.
\end{proof}

We now turn to Pseudojump Inversion, Theorem \ref{thm:pjinv}.

\begin{proof}[Proof of Theorem \ref{thm:pjinv}]
  Fix $e\in\NN$, and let $A$ be a real such that $A\geT0'$.  It will
  suffice to find a real $B$ such that $A\eqT J_e(B)\eqT B\oplus0'$.
  We shall construct a sequence of strings
  $\sigma_0\subseteq\sigma_1\subseteq
  \cdots\subseteq\sigma_j\subseteq\sigma_{j+1}\subseteq\cdots$ in
  $\NN^\ast$, and the desired real $B\in\NNNN$ will be obtained as the
  limit $\bigcup_j\sigma_j$ of this sequence.

  Stage $0$. Let $\sigma_0 = \la\ra$.

  Stage $2n+1$. Let $\sigma_{2n+1}=$ the least string
  $\sigma\supseteq\sigma_{2n}$ for which $n\in W_{e,|\sigma|}^\sigma$
  if such a $\sigma$ exists. Otherwise let
  $\sigma_{2n+1}=\sigma_{2n}$.

  Stage $2n+2$.  Let $\sigma_{2n+2}=\sigma_{2n+1}\cat\la A(n)\ra$.

  We now show that $B=\bigcup_j\sigma_j$ has the desired properties.
  We begin with the following observations.
  \begin{itemize}
  \item If $\sigma_{2n}$ is known, then in stage $2n+1$ the question
    of whether $\sigma$ exists can be answered recursively in $0'$
    (because the question is $\Sigma^0_1$), or recursively in $J_e(B)$
    (because $\sigma$ exists if and only if $n\in W_e^B$), hence also
    recursively in $A$ (because by hypothesis $0'\leT A$), or
    recursively in $B\oplus 0'$.  And once the existence or
    nonexistence of $\sigma$ is known, $\sigma_{2n+1}$ can be found
    recursively.
  \item If $\sigma_{2n+1}$ is known, then $\sigma_{2n+2}$ can be found
    recursively in $A$ (by the definition of $\sigma_{2n+2}$), or
    recursively in $B$ (because
    $\sigma_{2n+2}=\sigma_{2n+1}\cat\la
    B(|\sigma_{2n+1}|\ra$), hence also recursively in $J_e(B)$
    (because $B\leT J_e(B)$), or in $B\oplus0'$.
  \end{itemize}
  Combining these observations, we see that the entire sequence
  $\la\sigma_j\mid j\in\NN\ra$ is $\leT J_e(B)$,
  and $\leT B\oplus0'$, and $\leT A$.  We also have:
  \begin{itemize}
  \item $B\oplus0'\leT A$, because $B=\bigcup_j\sigma_j\leT A$ and by
    hypothesis $0'\leT A$.
  \item $J_e(B)\leT\la\sigma_j\mid j\in\NN\ra$, because $n\in W_e^B$
    if and only if $n\in W_{e,|\sigma_{2n+1}|}^{\sigma_{2n+1}}$.
  \item $A\leT\la\sigma_j\mid j\in\NN\ra$, because
    $A(n)=\sigma_{2n+2}(|\sigma_{2n+1}|)$ for all $n$.
  \end{itemize}
  Thus $A\leT\la\sigma_j\mid j\in\NN\ra\leT
  B\oplus0'\leT A$ and $J_e(B)\leT\la\sigma_j\mid
  j\in\NN\ra\leT J_e(B)$, hence $A\eqT
  J_e(B)\eqT B\oplus0'$, Q.E.D.
\end{proof}

\section{Completions of Peano Arithmetic}

Let $\PA$ denote Peano Arithmetic, and let $\Sent_\PA$ denote the set
of sentences of the language of $\PA$.  An \emph{extension of $\PA$}
is a set $T\subseteq\Sent_\PA$ which includes the axioms of $\PA$ and
is closed under logical consequence.  For any such $T$, a
\emph{completion of $T$} is an extension $X$ of $\PA$ which includes
$T$ and such that for each $\varphi\in\Sent_\PA$ exactly one of the
sentences $\varphi$ and $\lnot\,\varphi$ belongs to $X$.  Given an
extension $T$ of $\PA$, let $\Comp_T$ denote the set of all
completions of $T$.  By Lindenbaum's Lemma, $\Comp_T\ne\emptyset$ if
and only if $T$ is consistent.

Fix a primitive recursive G\"odel numbering
$\#:\varphi\mapsto\#(\varphi):\Sent_\PA\to\NN$.  For convenience we
shall assume that $\#$ is a one-to-one correspondence between
$\Sent_\PA$ and $\NN$.  This induces a one-to-one correspondence
between subsets $X$ of $\Sent_\PA$ and subsets of $\NN$, namely
$X\mapsto\{\#(\varphi)\mid\varphi\in X\}$.  And of course subsets of
$\NN$ are identified with their characteristic functions in $\oiNN$.
Therefore, letting $\CPA$ be the subset of $\oiNN$ corresponding to
$\Comp_\PA$, we see that $\CPA$ is a special $\Pi^0_1$ subclass of
$\oiNN$.

The purpose of this section is to prove that $\CPA$ has the Join
Property and the Pseudojump Inversion Property.  The Join Property
will follow easily from Theorem \ref{thm:join} plus other known
results, but for the Pseudojump Inversion Property we shall use a
construction involving the G\"odel-Rosser Incompleteness Theorem
\cite{mendelson}.

\begin{dfn}[$\PA$-degrees]
  By a \emph{$\PA$-degree} we mean the Turing degree $\degT(X)$ of
  some $X\in\CPA$.  A set $S$ of Turing degrees is said to be
  \emph{upwardly closed} if for all Turing degrees $\aaa$ and $\bbb$
  such that $\aaa\le\bbb$, $\aaa\in S$ implies $\bbb\in S$.
\end{dfn}

\begin{lem}\label{lem:solovay}
  The set of all $\PA$-degrees is upwardly closed.
\end{lem}

\begin{proof}
  This result is originally due to Robert M. Solovay.  For a proof see
  \cite[Theorem 2.21.3]{do-hi-book} or \cite[Corollary 6.6]{massrand}.
\end{proof}

\begin{lem}\label{lem:upward-join}
  Let $P\subseteq\oiNN$ be a nonempty $\Pioi$ class such that
  $\{\degT(X)\mid X\in P\}$ is upwardly closed.  Then $P$ has the Join
  Property.
\end{lem}

\begin{proof}
  Let $A$ and $Z$ be reals such that $A\geT Z\oplus0'$ and $Z\gT0$.
  By Theorem \ref{thm:join} there exists a real $C$ such that $A\eqT
  C'\eqT C\oplus0'\eqT C\oplus Z$.  By the Low Basis Theorem
  \cite[Theorem 2.1]{js} relativized to $C$, there exists $B_0\in P$
  such that $C'\eqT(B_0\oplus C)'\eqT B_0\oplus C\oplus0'$.  Since
  $B_0\leT B_0\oplus C$, we can find $B\in P$ such that $B\eqT
  B_0\oplus C$, hence $C'\eqT B'\eqT B\oplus0'$.  We now have $C'\eqT
  C\oplus Z\leT B\oplus Z\leT C'\oplus Z\eqT C'$, hence $A\eqT C'\eqT
  B'\eqT B\oplus0'\eqT B\oplus Z$, Q.E.D.
\end{proof}

\begin{thm}\label{thm:join-cpa}
  $\CPA$ has the Join Property.
\end{thm}

\begin{proof}
  This is immediate from Lemmas \ref{lem:solovay} and
  \ref{lem:upward-join}.
\end{proof}

Next we shall prove that $\CPA$ has the Pseudojump Inversion Property.
The proof will be presented in terms of completions of recursively
axiomatizable theories.  An extension $T$ of $\PA$ is said to be
\emph{recursively axiomatizable} if there is a recursive set of
sentences $S\subseteq\Sent_\PA$ such that $T$ is the closure of $S$
under logical consequence.  In this case $\Comp_T$ is clearly a
$\Pi^0_1$ subclass of $\CPA$.  The converse also holds:

\begin{lem}
  \label{lem:comp}
  We have an effective one-to-one correspondence $T\mapsto\Comp_T$
  between recursively axiomatizable extensions of $\PA$ and $\Pi^0_1$
  subclasses of $\CPA$.
\end{lem}

\begin{proof}
  Given a $\Pi^0_1$ class $P\subseteq\CPA$, we shall exhibit a
  recursively axiomatizable extension $T$ of $\PA$ such that
  $\Comp_T=P$.  Namely, let $T=\{\varphi\in\Sent_\PA\mid
  X(\#(\varphi))=1$ for all $X\in P\}$.  To see that $T$ is closed
  under logical consequence, it suffices to note that each $X\in P$
  belongs to $\CPA$ and is therefore closed under logical consequence.
  We also have $\PA\subseteq T$, because $P\subseteq\CPA$.  Moreover
  $\Comp_T=P$ by Lindenbaum's Lemma.  From the effective compactness
  of $P$ we see that $T$ is recursively enumerable.  Hence $T$ is
  recursively axiomatizable, and from an index of $P$ as a $\Pi^0_1$
  class we can compute an index for a recursive axiomatization of $T$.
  This completes the proof.
\end{proof}

\begin{rem}
  The idea of Lemma \ref{lem:comp} applies much more generally.  Given
  any language $L$ and any $L$-theory $T$, an \emph{extension of $T$}
  is any set $\widetilde{T}\subseteq\Sent_L$ which includes $T$ and is
  closed under logical consequence.  Regarding $\Comp_T$ as a closed
  set in the product space $\oi^{\Sent_L}$, we have a one-to-one
  correspondence $\widetilde{T}\mapsto\Comp_{\widetilde{T}}$ between
  extensions of $T$ and closed subsets of $\Comp_T$.
\end{rem}

Recall from \S1 that $\la P_i\mid i\in\NN\ra$ is a
fixed standard recursive enumeration of the $\Pi^0_1$ subclasses of
$\oiNN$.

\begin{lem}[splitting property]
  \label{lem:splitting}
  There is a primitive recursive function
  $s:\NN\times\oi\to\NN$ such that for all
  $i\in\NN$, if $P_i$ is a nonempty $\Pi^0_1$ subclass of
  $\CPA$ then $P_{s(i,0)}$ and $P_{s(i,1)}$ are nonempty disjoint
  $\Pi^0_1$ subclasses of $P_i$.
\end{lem}

\begin{proof}
  Lemma \ref{lem:comp} tell us that, given $i\in\NN$, we can
  effectively find a recursively axiomatizable extension $T$ of $\PA$
  such that $\Comp_T$ corresponds to $P_i\cap\CPA$.  We then apply the
  G\"odel-Rosser Incompleteness Theorem \cite{mendelson} to
  effectively find a sentence $\psi\in\Sent_\PA$ such that if $T$ is
  consistent, then neither $\psi$ nor $\lnot\,\psi$ belongs to $T$.
  Let $T_1$ (respectively $T_0$) be the closure of $T\cup\{\psi\}$
  (respectively $T\cup\{\lnot\,\psi\}$) under logical consequence.
  Clearly $\Comp_{T_1}$ and $\Comp_{T_0}$ are disjoint, and they are
  nonempty if $\Comp_T$ is nonempty.  Apply Lemma \ref{lem:comp} to
  effectively find $s(i,1),s(i,0)\in\NN$ such that $P_{s(i,1)}$ and
  $P_{s(i,0)}$ correspond to $\Comp_{T_1}$ and $\Comp_{T_0}$
  respectively.  This completes the proof.
\end{proof}

We shall now use this splitting property to redo Theorem
\ref{thm:pjinv} within $\CPA$.

\begin{thm}\label{thm:pjinv-cpa}
  $\CPA$ has the Pseudojump Inversion Property.
\end{thm}

\begin{proof}
  Fix $e\in\NN$, and let $A\in\oiNN$ be a real such that $A\geT0'$.
  It will suffice to find a real $B\in\CPA$ such that $A\eqT
  J_e(B)\eqT B\oplus0'$.  To find $B$, we shall construct a sequence
  of $\Pioi$ classes $Q_0\supseteq Q_1\supseteq\cdots\supseteq
  Q_j\supseteq Q_{j+1}\supseteq\cdots$, by induction on $j$ starting
  with $Q_0=\CPA\subseteq\oiNN$.  The intersection $\bigcap_jQ_j$ will
  be nonempty, and we shall have $B\in\bigcap_jQ_j$.  At the same
  time, we shall construct a function $f:\NN\to\NN$ such that
  $Q_j=P_{f(j)}$ for all $j$.

  Fix a primitive recursive splitting function $s$ as in Lemma
  \ref{lem:splitting}.

  Stage $0$.  Fix $f(0)\in\NN$ such that $P_{f(0)}=Q_0=\CPA$.

  Stage $2n+1$.  By induction we have $Q_{2n}=P_{f(2n)}$.  If
  $P_{f(2n)}\cap\{X\mid n\notin J_e(X)\}=\emptyset$ let
  $Q_{2n+1}=P_{f(2n)}$ and $f(2n+1)=f(2n)$.  Otherwise, let
  $Q_{2n+1}=P_{f(2n)}\cap\{X\mid n\notin J_e(X)\}$ and choose
  $f(2n+1)$ so that $P_{f(2n+1)}=P_{f(2n)}\cap\{X\mid n\notin
  J_e(X)\}$.  This $f(2n+1)$ is found primitive recursively from
  $f(2n)$ and $n$ and $e$.

  Stage $2n+2$.  By induction we have $Q_{2n+1}=P_{f(2n+1)}$.  Let
  $f(2n+2)=s(f(2n+1),A(n))$ and $Q_{2n+2}=P_{s(f(2n+1),A(n))}$, where
  $s$ is our splitting function.  This $f(2n+2)$ is found primitive
  recursively from $f(2n+1)$ and $A(n)$.

  This completes the construction.

  By construction each $Q_j$ is nonempty, so by compactness of
  $\oiNN$ their intersection $\bigcap_jQ_j$ is nonempty.  It remains
  to show that $B\in\bigcap_jQ_j$ has the desired properties.  We
  begin with the following observations.
  \begin{itemize}
  \item If $f(2n)$ is known then $f(2n+1)$ can be
    found recursively in $0'$ (by checking whether the $\Pi^0_1$ class
    $P_{f(2n)}\cap\{X\mid n\notin J_e(X)\}$ is empty or not), or
    recursively in $J_e(B)$ (by checking whether $n\in J_e(B)$ or
    not).
  \item If $f(2n+1)$ is known then $f(2n+2)$ can be found recursively
    in $B$ (by finding the $k\in\oi$ such that $B\notin
    P_{s(f(2n+1),k)}$), or recursively in $A$ (by evaluating
    $s(f(2n+1),A(n))$).
  \end{itemize}
  Combining these observations, we see that $f$ is $\leT
  A\oplus0'\eqT A$, and $\leT J_e(B)$, and $\leT
  B\oplus0'$.  Conversely, we also have:
  \begin{itemize}
  \item $A\leT f$, because $A(n)=i$ if and only if
    $f(2n+2)=s(f(2n+1),i)$.
  \item $J_e(B)\leT f$, because $n\in J_e(B)$ if and only if
    $f(2n+1)=f(2n)$.
  \item $B\oplus0'\leT f$, because $B\leT J_e(B)\leT
    f$ and $0'\leT A\leT f$.
  \end{itemize}
  Thus $f\eqT A\eqT J_e(B)\eqT B\oplus 0'$ and the proof is complete.
\end{proof}

\begin{rem}
  We have used the Splitting Lemma \ref{lem:splitting} to prove that
  $\CPA$ has the Pseudojump Inversion Property.  Similarly, it would
  be possible to use the Splittng Lemma to prove directly that $\CPA$
  has the Join Property.  This direct proof would be in contrast to
  our shorter but indirect proof in Theorem \ref{thm:join-cpa} above.
  Note however that a version of the Splitting Lemma is used in our
  proof of Solovay's Lemma \ref{lem:solovay}; see \cite[\S6]{massrand}
  and \cite[\S3]{pizowkl}.
\end{rem}

\section{Turing degree isomorphism}

\begin{dfn}[Turing degree isomorphism]
  Following \cite{js2} we say that $P,Q\subseteq\NNNN$ are
  \emph{Turing degree isomorphic} if $\{\degT(X)\mid X\in
  P\}=\{\degT(X)\mid X\in Q\}$.
\end{dfn}

The purpose of this section is to prove that the Join Property and the
Pseudojump Inversion Property hold for all $\Pi^0_1$ subclasses of
$\oiNN$ which are Turing degree isomorphic to $\CPA$.  We also obtain
some more general-looking results.

We begin with the Join Property.  

\begin{thm}\label{thm:iso-join}
  Let $P\subseteq\oiNN$ be a $\Pi^0_1$ class.  If $P$ is Turing degree
  isomorphic to $\CPA$, then $P$ has the Join Property.
\end{thm}

\begin{proof}
  Let $A$ and $Z$ be reals such that $A\geT Z\oplus0'$ and $Z\gT0$.
  By Theorem \ref{thm:join-cpa} let $B\in\CPA$ be such that $A\eqT
  B'\eqT B\oplus0'\eqT B\oplus Z$.  Since $P$ is Turing degree
  isomorphic to $\CPA$, let $C\in P$ be such that $B\eqT C$.  Then
  $A'\eqT C'\eqT C\oplus0'\eqT C\oplus Z$.  Thus $P$ has the Join
  Property, Q.E.D.
\end{proof}

\begin{rem}\label{rem:emb-join}
  More generally, for $P,Q\subseteq\NNNN$ we say that $P$ is
  \emph{Turing degree embeddable} into $Q$ if $\{\degT(X)\mid X\in
  P\}\subseteq\{\degT(X)\mid X\in Q\}$.  The proof of Theorem
  \ref{thm:iso-join} shows that if $P$ has the Join Property and is
  Turing degree embeddable into $Q$, then $Q$ has the Join Property.
\end{rem}

We now turn to Pseudojump Inversion.  Unfortunately, we cannot simply
imitate the proof of Theorem \ref{thm:iso-join}.  This is because
pseudojump operators are not invariant under Turing equivalence, i.e.,
$X\eqT Y$ typically does not imply $J_e(X)\eqT J_e(Y)$.  Consequently,
we do not know whether the Pseudojump Inversion Property is invariant
under Turing degree isomorphism of special $\Pioi$ subclasses of
$\oiNN$.  However, we shall settle some interesting special cases of
this question.  As a first step, consider the following notion, which
is a special case of Turing degree isomorphism.

\begin{dfn}[recursive homeomorphism]
  For $P,Q\subseteq\NNNN$, a \emph{recursive homeomorphism} of $P$
  onto $Q$ is a one-to-one onto mapping $\Phi:P\to Q$ such that both
  $\Phi$ and its inverse $\Phi^{-1}:Q\to P$ are the restrictions of
  partial recursive functionals to $P$ and $Q$ respectively.  We say
  that $P$ and $Q$ are \emph{recursively homeomorphic} if there exists
  a recursive homeomorphism of $P$ onto $Q$.
\end{dfn}

\begin{lem}\label{lem:homeo-pjinv}
  Let $P,Q\subseteq\NNNN$ be recursively homeomorphic.  If $P$ has the
  Pseudojump Inversion Property, then so does $Q$.
\end{lem}

\begin{proof}
  Assume that $P$ has the Pseudojump Inversion Property.  We shall
  prove that $Q$ has the Pseudojump Inversion Property.  Given
  $e\in\NN$ and $A\geT0'$, it will suffice to find $C\in Q$ such that
  $A\eqT J_e(C)\eqT C\oplus0'$.

  Let $\Phi:P\to Q$ be a recursive homeomorphism.  Let $i\in\NN$ be
  such that for all $X\in P$ and all $n\in\NN$ we have
  $\varphi_i^X(n)\simeq\varphi_e^{\Phi(X)}(n)$.  Then for all $X\in P$
  we have $W_i^X=\{n\in\NN\mid\varphi_i^X(n)\downarrow\}=
  \{n\in\NN\mid\varphi_e^{\Phi(X)}(n)\downarrow\}=W_e^{\Phi(X)}$.
  Since $P$ has the Pseudojump Inversion Property, let $B\in P$ such
  that $A\eqT J_i(B)\eqT B\oplus0'$.  Let $C =\Phi(B)$.  Because
  $\Phi$ is a recursive homeomorphism, we have $B\eqT\Phi(B)=C$, hence
  $J_e(C)=C\oplus W_e^C=C\oplus W_e^{\Phi(B)}=C\oplus W_i^B\eqT
  B\oplus W_i^B = J_i(B)$, hence $A\eqT J_e(C)\eqT C\oplus0'$, Q.E.D.
\end{proof}

As a bridge from Turing degree isomorphism to recursive homeomorphism,
we have the following lemma.

\begin{lem}\label{lem:bridge}
  Let $P,Q\subseteq\oiNN$ be nonempty $\Pi^0_1$ classes.  If $P$ is
  Turing degree embeddable into $Q$, then there exist nonempty
  $\Pi^0_1$ subclasses $\widetilde{P}\subseteq P$ and
  $\widetilde{Q}\subseteq Q$ such that $\widetilde{P}$ and
  $\widetilde{Q}$ are recursively homeomorphic.
\end{lem}

\begin{proof}
  We shall draw on some facts about \emph{hyperimmune-freeness} and
  \emph{truth-table reducibility}.  For this background, see
  \cite[\S\S8.3,9.6]{rogers} and \cite[\S4]{massrand}.

  By the Hyperimmune-Free Basis Theorem \cite[Theorem 2.4]{js} (see
  also \cite[Theorem 4.19]{massrand}), let $X_0\in P$ be
  hyperimmune-free.  Since $P$ is Turing degree embeddable into $Q$,
  let $Y_0\in Q$ be such that $X_0\eqT Y_0$.  Then $Y_0$ is also
  hyperimmune-free, and by hyperimmune-freeness there exist
  truth-table functionals $\Phi,\Psi:\oiNN\to\oiNN$ such that
  $\Phi(X_0)=Y_0$ and $\Psi(Y_0)=X_0$.  Let $\widetilde{P}=\{X\in
  P\mid\Phi(X)\in Q$ and $\Psi(\Phi(X))=X\}$.  Then $\widetilde{P}$ is
  a $\Pi^0_1$ subclass of $P$, and it is nonempty because it contains
  $X_0$.  Moreover $\widetilde{Q}=\{\Phi(X)\mid X\in\widetilde{P}\}$
  is also a $\Pi^0_1$ subclasss of $Q$, and we have a recursive
  homeomorphism $\Phi\res\widetilde{P}$ of $\widetilde{P}$ onto
  $\widetilde{Q}$.
\end{proof}

\begin{lem}\label{lem:cpa-tilde}
  Any nonempty $\Pioi$ subclass of $\CPA$ is recursively homeomorphic
  to $\CPA$.
\end{lem}

\begin{proof}
  See \cite[\S3]{pizowkl}.
\end{proof}

We can now prove the following theorem, which is parallel to Theorem
\ref{thm:iso-join}.

\begin{thm}\label{thm:iso-pjinv}
  Let $P\subseteq\oiNN$ be a $\Pi^0_1$ class.  If $P$ is Turing degree
  isomorphic to $\CPA$, then $P$ has the Pseudojump Inversion
  Property.
\end{thm}

\begin{proof}
  Let $P$ be Turing degree isomorphic to $\CPA$.  By Lemma
  \ref{lem:bridge} there are nonempty recursively homeomorphic
  $\Pi^0_1$ classes $\widetilde{P}\subseteq P$ and
  $\widetilde{\CPA}\subseteq\CPA$.  By Lemma \ref{lem:cpa-tilde}
  $\widetilde{\CPA}$ is recursively homeomorphic to $\CPA$.  By
  Theorem \ref{thm:pjinv-cpa} $\CPA$ has the Pseudojump Inversion
  Property, so by Lemma \ref{lem:homeo-pjinv} $\widetilde{\CPA}$ and
  hence $\widetilde{P}$ have the Pseudojump Inversion Property.  But
  then, since $\widetilde{P}\subseteq P$, it follows that $P$ has the
  Pseudojump Inversion Property, Q.E.D.
\end{proof}

Next we prove a more general-looking result.  Recall from
\cite{massrand,extre,massaed,massmtr,dou-tutorial} that a nonempty
$\Pioi$ class $P\subseteq\oiNN$ is said to be \emph{Muchnik complete}
if every nonempty $\Pioi$ class $Q\subseteq\oiNN$ is \emph{Muchnik
  reducible} to $P$, i.e., for all $X\in P$ there exists $Y\in Q$ such
that $Y\leT X$.

\begin{lem}\label{lem:Muchnik}
  A $\Pioi$ class $P\subseteq\oiNN$ is Muchnik complete if and only if
  it is Turing degree isomorphic to $\CPA$.  Moreover, for such a $P$
  the set of Turing degrees $\{\deg_T(X)\mid X\in P\}$ is upwardly
  closed.
\end{lem}

\begin{proof}
  See \cite[\S\S3,6]{massrand}.
\end{proof}

\begin{thm}
  Let $P\subseteq\oiNN$ be a nonempty $\Pioi$ class.  If $P$ is
  Muchnik complete, then $P$ has the Join Property and the Pseudojump
  Inversion Property.
\end{thm}

\begin{proof}
  This is immediate from Lemma \ref{lem:Muchnik} and Theorems
  \ref{thm:iso-join} and \ref{thm:iso-pjinv}.
\end{proof}

An even more general-looking result reads as follows.

\begin{thm}
  Let $P\subseteq\oiNN$ be a nonempty $\Pioi$ class such that
  $\{\deg_T(Y)\mid Y\in P\}$ is upwardly closed.  Then any
  $P_1\subseteq\NNNN$ which includes $P$ has the Join Property and the
  Pseudojump Inversion Property.
\end{thm}

\begin{proof}
  By Lemma \ref{lem:Muchnik} $\CPA$ is Turing degree embeddable into
  $P$.  Our result then follows by Remark \ref{rem:emb-join} and
  Theorems \ref{thm:iso-join} and \ref{thm:iso-pjinv}.
\end{proof}

\section{$\Pi^0_1$ classes of positive measure}

For $\sigma\in\oist$ we write
$\llb\sigma\rrb=\{X\in\oiNN\mid\sigma\subset X\}$.  Let $\mu$ be the
\emph{fair coin measure} on $\oiNN$, defined by letting
$\mu(\llb\sigma\rrb)=2^{-|\sigma|}$ for all $\sigma\in\oist$.  This
$\mu$ is a Borel probability measure on $\oiNN$.  A $\Pioi$ class
$P\subseteq\oiNN$ is said to be \emph{of positive measure} if
$\mu(P)>0$.  In this section we note that such a $P$ must have the
Pseudojump Inversion Property but need not have the Join Property.  We
also obtain the same results for $\Pi^0_1$ subclasses of $\oiNN$ which
are Turing degree isomorphic to such a $P$.

To prove these results we need some basic facts about
\emph{Martin-L\"of randomness} and \emph{$\LR$-reducibility}.  We cite
our semi-expository papers \cite{massrand,aedsh,massaed,massmtr} but
one can also consult the treatises of Downey and Hirschfeldt
\cite{do-hi-book} and Nies \cite{nies-book}.  Let
$\MLR=\{X\in\oiNN\mid X$ is Martin-L\"of random$\}$.

\begin{lem}\label{lem:pos}
  A $\Pi^0_1$ subclass of $\oiNN$ is of positive measure if and only
  if it includes a nonempty $\Pi^0_1$ subclass of $\MLR$.
\end{lem}

\begin{proof}
  This is because $\MLR\subseteq\oiNN$ is a $\Sigma^0_2$ class of
  full measure.  See for instance \cite[\S8]{massrand}.
\end{proof}

\begin{lem}[due to Nies \protect{\cite{nies-book}}] \label{lem:nies}
  Let $P\subseteq\oiNN$ be a $\Pi^0_1$ class of positive measure.
  Then $P$ has the Pseudojump Inversion Property.
\end{lem}

\begin{proof}
  See \cite[Theorem 6.3.9]{nies-book} or a simpler proof in
  \cite[Theorem 5.1]{aedsh}.
\end{proof}

\begin{thm}
  Let $P\subseteq\oiNN$ be a $\Pi^0_1$ class of positive measure,
  and let $Q\subseteq\oiNN$ be a $\Pi^0_1$ class which is Turing
  degree isomorphic to $P$.  Then $Q$ has the Pseudojump Inversion
  Property.
\end{thm}

\begin{proof}
  By Lemma \ref{lem:pos} let $P_0$ be a nonempty $\Pi^0_1$ subclass of
  $P\cap\MLR$.  Then $P_0$ is Turing degree embeddable in $Q$, so by
  Lemma \ref{lem:bridge} we can find nonempty $\Pi^0_1$ classes
  $\widetilde{P}\subseteq P_0$ and $\widetilde{Q}\subseteq Q$ which
  are recursively homeomorphic.  By Lemma \ref{lem:pos}
  $\widetilde{P}$ is of positive measure, so by Lemma \ref{lem:nies}
  $\widetilde{P}$ has the Pseudojump Inversion Property.  It then
  follows by Lemma \ref{lem:homeo-pjinv} that $\widetilde{Q}$ and
  hence $Q$ have the Pseudojump Inversion Property.
\end{proof}

Relativizing the notion of Martin-L\"of randomness, for any real $Y$
we write $\MLR^Y=\{X\in\oiNN\mid X$ is Martin-L\"of random relative
to $Y\}$.  If $\MLR^Y\subseteq\MLR^Z$ we say that $Z$ \emph{is
  $\LR$-reducible to} $Y$, abbreviated as $Z\le_\LR Y$.  Intuitively
this means that $Y$ has at least as much ``derandomizing power'' as
$Z$.  Clearly $Z\leT Y$ implies $Z\le_\LR Y$, but the converse
does not hold:

\begin{lem}\label{lem:z}
  There exists a nonrecursive real $Z$ such that $Z\le_\LR0$.
\end{lem}

\begin{proof}
  See \cite[Theorem 6.1]{aedsh} or \cite{do-hi-book,nies-book}.
\end{proof}

On the other hand, we have:

\begin{lem}\label{lem:low}
  If $Z\le_\LR Y$ then $Z'\leT Z\oplus Y'$.
\end{lem}
\begin{proof}
  See \cite[Theorem 8.8]{aedsh} or \cite{do-hi-book,nies-book}.
\end{proof}

\begin{lem}[Lutz \protect{\cite{lutz-190809}}]\label{lem:lutz}
  If $Z\le_\LR0$ then $X\oplus Z\le_\LR X$ for all $X\in\MLR$.
\end{lem}

\begin{proof}
  Assume $Z\le_\LR0$ and $X\in\MLR$.  We must prove that
  $\MLR^X\subseteq\MLR^{X\oplus Z}$.  Given $X_1\in\MLR^X$, it follows
  by Van Lambalgen's Theorem (\cite[Theorem 3.6]{aedsh}, see also
  \cite[Corollary 6.9.3]{do-hi-book}) that $X_1\oplus X\in\MLR$.  But
  then $X_1\oplus X\in\MLR^Z$, so by Van Lambalgen's Theorem relative
  to $Z$ we have $X_1\in\MLR^{X\oplus Z}$, Q.E.D.
\end{proof}

\begin{thm}[Lutz \protect{\cite{lutz-190809}}]\label{thm:lutz}
  \label{pi01 class of random reals counterexample to posner-robinson}
  $\MLR$ does not have the Join Property.
\end{thm}

\begin{proof}
  By Lemma \ref{lem:z} let $Z\gT0$ be such that $Z\le_\LR0$.  For any
  $X\in\MLR$ we have $X\oplus Z\le_\LR X$ by Lemma \ref{lem:lutz},
  hence $(X\oplus Z)'\leT X'\oplus Z$ by Lemma \ref{lem:low}.
  If $\MLR$ had the Join Property, there would be an $X\in\MLR$ such
  that $X'\eqT X\oplus Z$, hence $(X\oplus Z)'\leT X\oplus Z$, a
  contradiction.
\end{proof}

\begin{thm}[Lutz \protect{\cite{lutz-190809}}]
  There is a $\Pi^0_1$ class $P\subseteq\oiNN$ of positive measure
  which does not have the Join Property.
\end{thm}

\begin{proof}
  This is immediate from Theorem \ref{thm:lutz} and Lemma
  \ref{lem:pos}.
\end{proof}

\section{$\Pi^0_1$ classes constructed by priority arguments}

In this section we construct a special $\Pioi$ class $Q\subseteq\oiNN$
which has neither the Join Property nor the Pseudojump Inversion
Property.  In addition, this $\Pioi$ class $Q$ has some other
interesting features, which we also discuss.

We begin by stating a technical theorem which embodies our
construction of $Q$.  Recall from \S1 that $\la P_e\ra_{e\in\NN}$ is a
fixed enumeration of all $\Pioi$ classes.

\begin{thm}\label{thm:Q}
  There is a nonempty perfect\footnote{A topological space is said to
    be \emph{perfect} if it has no isolated points, i.e., there is no
    open set consisting of exactly one point.  In particular, a set
    $P\subseteq\oiNN$ is perfect if and only if there is no
    $\tau\in\oist$ such that $P\cap\llb\tau\rrb=\{X\}$ for some
    $X\in\oiNN$.} $\Pioi$ class $Q\subseteq\oiNN$ with the following
  properties.  For all $e,n\in\NN$ there exists $m\in\NN$ such that
  for all pairwise distinct $\tau_1,\ldots,\tau_n\in\oist$ with
  $|\tau_1|=\cdots=|\tau_n|\ge m$, the $\Pioi$ class
  \begin{center}
    $(Q\cap\llb\tau_1\rrb)\times\cdots\times(Q\cap\llb\tau_n\rrb)$
  \end{center}
  is either disjoint from $P_e$ or included in $P_e$.  Moreover, this
  $m$ can be computed from $e$ using an oracle for $0'$.
\end{thm}

Before proving Theorem \ref{thm:Q}, we spell out some of its
consequences which are of more general interest.  The next theorem
summarizes these features of $Q$.

\begin{thm}\label{thm:Q-props}
  Let $Q\subseteq\oiNN$ be a nonempty $\Pioi$ class as in Theorem
  \ref{thm:Q}.
  \begin{enumerate}
  \item $Q$ is \emph{thin}, i.e., for every $\Pioi$ subclass $P$ of
    $Q$, the complement $Q\setminus P$ is again a $\Pioi$ subclass of
    $Q$.  More generally, for any finite sequence $Q_1,\ldots,Q_n$ of
    pairwise disjoint\footnote{The hypothesis of pairwise disjointness
      is essential.  For instance, if $P=\{X\oplus X\mid X\in Q\}$
      then $(Q\times Q)\setminus P$ is not $\Pioi$, so $Q\times Q$ is
      not thin.} $\Pioi$ subclasses of $Q$, the $\Pioi$ class
    $Q_1\times\cdots\times Q_n$ is thin.
  \item $Q$ is \emph{special}, i.e., no $X\in Q$ is recursive.  More
    generally, the Turing degrees of members of $Q$ are
    \emph{independent}, i.e., no $X\in Q$ is $\leT$ the join of any
    finitely many other members of $Q$.
  \item Every $X\in Q$ is of minimal truth-table degree.
    Consequently, every hyperimmune-free $X\in Q$ is of minimal Turing
    degree.
  \item Every finite sequence $X_1,\ldots,X_n\in Q$ is
    \emph{generalized low}, i.e., $(X_1\oplus\cdots\oplus X_n)'\eqT
    X_1\oplus\cdots\oplus X_n\oplus0'$.
  \end{enumerate}
\end{thm}

\begin{proof}
  We prove parts 1 through 4 in that order.

  1. To see that $Q$ is thin, let $P$ be a $\Pi^0_1$ subclass of $Q$.
  Fix $e\in\NN$ such that $P=P_e$, and let $m$ be as in Theorem
  \ref{thm:Q}.  Then for each $\tau\in\oi^m$ the $\Pioi$ class
  $Q\cap\llb\tau\rrb$ is either disjoint from $P$ or included in $P$.
  Hence $Q\setminus P=\bigcup_\tau(Q\cap\llb\tau\rrb)$ where the union
  is taken over all $\tau\in\oi^m$ such that $Q\cap\llb\tau\rrb$ is
  disjoint from $P$.  Thus $Q\setminus P$ is is a union of finitely
  many $\Pioi$ classes, so it too is a $\Pioi$ class.

  For the generalization, let $P$ be a $\Pioi$ subclass of
  $Q_1\times\cdots\times Q_n$ where $Q_1,\ldots,Q_n$ are pairwise
  disjoint $\Pioi$ subclasses of $Q$.  Fix $e\in\NN$ such that $P=P_e$,
  and let $m$ be as in Theorem \ref{thm:Q}.  We then have
  \begin{center}
    $(Q_1\times\cdots\times Q_n)\setminus P=
    \bigcup_{\tau_1,\ldots,\tau_n}((Q_1\cap\llb\tau_1\rrb)
    \times\cdots\times(Q_n\cap\llb\tau_n\rrb))$
  \end{center}
  where the union is taken over all pairwise distinct sequences
  $\tau_1,\ldots,\tau_n\in\oi^m$ such that
  $(Q_1\cap\llb\tau_1\rrb)\times\cdots\times(Q_n\cap\llb\tau_n\rrb)$
  is disjoint from $P$.  Thus $(Q_1\times\cdots\times Q_n)\setminus P$
  is a union of finitely many $\Pioi$ classes, so it too is a $\Pioi$
  class, Q.E.D.

  2. Assume for a contradiction that $X\in Q$ is recursive.  Fix
  $e\in\NN$ such that $P_e=\{X\}$, let $m\in\NN$ be as in Theorem
  \ref{thm:Q}, and let $\tau=X\res m$.  Clearly $X\in
  Q\cap\llb\tau\rrb$, hence $Q\cap\llb\tau\rrb$ is not disjoint from
  $P_e$, hence $Q\cap\llb\tau\rrb$ is included in $P_e$, hence $X$ is
  the unique member of $Q\cap\llb\tau\rrb$.  Thus $X$ is an isolated
  point of $Q$, contradicting the fact that $Q$ is perfect.

  For the generalization, it will suffice to show that $X\nleT
  X_1\oplus\cdots\oplus X_n$ for all pairwise distinct
  $X,X_1,\ldots,X_n\in Q$.  Let $\Psi:{\subseteq}\oiNN\to\oiNN$ be a
  partial recursive functional, and assume for a contradiction that
  $X=\Psi(X_1\oplus\cdots\oplus X_n)$.  Let $e\in\NN$ be such that
  \begin{center}
    $P_e=\{Y\oplus Y_1\oplus\cdots\oplus Y_n\mid\forall i\,\forall
    j\,($if $\Psi(Y_1\oplus\cdots\oplus Y_n)(i){\downarrow}=j$ then
    $Y(i)=j)\}$.
  \end{center}
  Let $m$ be as in Theorem \ref{thm:Q} and sufficiently large so that
  $\tau=X\res m$, $\tau_1=X_1\res m$, \ldots, $\tau_n=X_n\res m$ are
  pairwise distinct.  The $\Pioi$ class
  \begin{center}
    $(Q\cap\llb\tau\rrb)\times(Q\cap\llb\tau_1\rrb)\times
    \cdots\times(Q\cap\llb\tau_n\rrb)$
  \end{center}
  contains $X\oplus X_1\oplus\cdots\oplus X_n$ and is therefore not
  disjoint from $P_e$, so it is included in $P_e$.  In particular, for
  all $Y\in Q\cap\llb\tau\rrb$ and all $i\in\NN$ we have
  $\Psi(X_1\oplus\cdots\oplus X_n)(i){\downarrow}=X(i)=Y(i)$, hence
  $X=Y$.  Thus $Q\cap\llb\tau\rrb=\{X\}$, so again $X$ is an isolated
  point of $Q$, contradicting the fact that $Q$ is perfect.

  3. Let $X\in Q$ be given.  We have already seen that $X$ is not
  recursive.  To prove that $X$ is of miminal truth-table degree, it
  remains to show that $X\lett\Psi(X)$ for any truth-table functional
  $\Psi$ such that $\Psi(X)$ is not recursive.  Given such a
  functional $\Psi$, let $e$ be such that
  \begin{center}
    $P_e=\{X_0\oplus X_1\in Q\times Q\mid\Psi(X_0)=\Psi(X_1)\}$.
  \end{center}
  By Theorem \ref{thm:Q} let $m$ be such that for all $\tau\in\oi^{\ge
    m}$ the $\Pioi$ class $(Q\cap\llb\tau\cat\la0\ra\rrb)\times
  (Q\cap\llb\tau\cat\la1\ra\rrb)$ is either disjoint from $P_e$ or
  included in $P_e$.  If it is disjoint from $P_e$, let us say that
  $\tau$ is \emph{splitting}.

  Case 1: For all sufficiently large $\tau\subset X$, $\tau$ is
  splitting.  Then for all sufficiently large $\tau\subset X$ and all
  $X_0\in Q\cap\llb\tau\cat\la0\ra\rrb$ and $X_1\in
  Q\cap\llb\tau\cat\la1\ra\rrb$, we have $\Psi(X_0)\ne\Psi(X_1)$.  In
  particular we have $\Psi(X)\ne\Psi(Y)$ for all $Y\in
  Q\cap\llb\tau\rrb$ such that $X\ne Y$.  From this it follows that
  $X$ is truth-table reducible to $\Psi(X)$.

  Case 2: There are arbitrarily large $\tau\subset X$ such that $\tau$
  is not splitting.  Let $\tau\subset X$ be non-splitting with
  $|\tau|\ge m$.  Then $Q\cap\llb\tau\cat\la0\ra\rrb$ and
  $Q\cap\llb\tau\cat\la1\ra\rrb$ are nonempty, and for all $X_0\in
  Q\cap\llb\tau\cat\la0\ra\rrb$ and $X_1\in
  Q\cap\llb\tau\cat\la1\ra\rrb$ we have $\Psi(X_0)=\Psi(X_1)$.  In
  particular, letting $i,j\in\oi$ be such that $X\in
  Q\cap\llb\tau\cat\la i\ra\rrb$ and $i+j=1$, we have
  $\Psi(X)=\Psi(Y)$ for all $Y\in Q\cap\llb\tau\cat\la j\ra\rrb$. From
  this it follows that $\Psi(X)$ is recursive, Q.E.D.

  A general property of hyperimmune-free reals $X$ is that the Turing
  degrees $\le\degT(X)$ are the same as the truth-table degrees
  $\le\degtt(X)$.  Since every $X\in Q$ is of minimal truth-table
  degree, it follows that every hyperimmune-free $X\in Q$ is of
  minimal Turing degree.

  4. Up until now we have not used the ``moreover'' clause of Theorem
  \ref{thm:Q}, but now we shall use it.  We may safely assume that
  $X_1,\ldots,X_n\in Q$ are pairwise distinct.  Recall that we have
  defined the Turing jump $X'\in\oiNN$ of $X\in\oiNN$ to be (the
  characteristic function of) the set $\{e\in\NN\mid X\notin P_e\}$.
  To compute $(X_1\oplus\cdots\oplus X_n)'$ we proceed as follows.
  First we use our oracle for $X_1\oplus\cdots\oplus X_n$ to find
  $l\in\NN$ such that $X_1\res l$, \ldots, $X_n\res l$ are pairwise
  distinct.  Then, given $e\in\NN$, we use our oracle for $0'$ to
  compute $m\ge l$ as in Theorem \ref{thm:Q}.  Letting $\tau_i=X_i\res
  m$ for $i=1,\ldots,n$, we know that the $\Pioi$ class
  $(Q\cap\llb\tau_1\rrb)\times\cdots\times(Q\cap\llb\tau_n\rrb)$ is
  either disjoint from $P_e$ or included in $P_e$.  Using our oracle
  for $0'$ again, we can decide whether this $\Pioi$ class is disjoint
  from $P_e$ or not.  The answer to this question tells us whether
  $e\in(X_1\oplus\cdots\oplus X_n)'$ or not.  Thus
  $(X_1\oplus\cdots\oplus X_n)'$ is computable from
  $X_1\oplus\cdots\oplus X_n\oplus0'$, Q.E.D.
\end{proof}

\begin{rem}
  A plausible generalization of part 3 of Theorem \ref{thm:Q-props}
  would say that for every pairwise distinct finite sequence
  $X_1,\ldots,X_n\in Q$, the tt-degrees
  $\le\degtt(X_1\oplus\cdots\oplus X_n)$ should form a lattice
  isomorphic to the powerset of $\{1,\ldots,n\}$.  From this it would
  follow that the same holds for Turing degrees, provided
  $X_1\oplus\cdots\oplus X_n$ is hyperimmune-free.
\end{rem}

\begin{thm}\label{thm:Q-nonprops}
  The special $\Pioi$ class $Q\subseteq\oiNN$ of Theorem \ref{thm:Q}
  has neither the Join Property nor the Pseudojump Inversion Property.
\end{thm}

\begin{proof}
  We rely mainly on part 4 of Theorem \ref{thm:Q-props}.  To see that
  the Join Property fails for $Q$, fix $Z\in Q$.  Then $Z\gT0$ but for
  all $B\in Q$ we have $(B\oplus Z)'\eqT B\oplus Z\oplus0'$, hence
  $0'\nleT B\oplus Z$.  To see that Pseudojump Inversion fails for
  $Q$, consider a pseudojump operator $J_e$ with the
  property\footnote{The existence of such pseudojump operators is well
    known \cite[\S VII.1]{soare}.  One way to obtain such an operator
    is to combine part 4 of Theorem \ref{thm:Q-props} with the R.\ E.\
    Basis Theorem \cite[Theorem 3]{js2} to get an $e\in\NN$ such that
    $0\lT W_e$ and $(W_e)'\eqT0'$.  The desired operator $J_e$ is then
    obtained by uniform relativization to an arbitrary Turing oracle
    $X$.} that $X\lT J_e(X)$ and $(J_e(X))'\eqT X'$ for all
  $X\in\NNNN$.  Then for all $B\in Q$ we have $(J_e(B))'\eqT B'\eqT
  B\oplus0'$, hence $0'\nleT J_e(B)$.
\end{proof}

The rest of this section is devoted to the proof of Theorem
\ref{thm:Q}.  We use a priority construction in the vein of
Martin/Pour-El \cite{mpe} and Jockusch/Soare \cite[Theorem 4.7]{js}.

Our construction will be presented in terms of treemaps.  A
\emph{treemap} is a mapping $T:\oist\to\oist$ such that
$T(\sigma\cat\la i\ra)\supseteq T(\sigma)\cat\la i\ra$ for all
$\sigma\in\oist$ and all $i\in\oi$.  Note that for any treemap $T$ and
$\sigma,\rho\in\oist$ we have $\sigma\subset\rho$ if and only if
$T(\sigma)\subset T(\rho)$.  Note also that there is a one-to-one
correspondence between treemaps $T$ and nonempty perfect closed
subsets of $\oiNN$, given by $T\mapsto[T]=\{T(X)\mid X\in\oiNN\}$
where $T(X)=\bigcup_{n\in\NN}T(X\res n)$.

Recall that any $\Pioi$ class $P\subseteq\oiNN$ is closed.  Therefore,
if $P$ is also nonempty and perfect, there is a unique treemap $T_P$
such that $P=[T_P]$.

\begin{lem}\label{lem:treemap}
  Let $P\subseteq\oiNN$ be a nonempty perfect $\Pioi$ class.  Then,
  the treemap $T=T_P$ corresponding to $P$ is computable using $0'$ as
  an oracle.
\end{lem}

\begin{proof}
  Given $\tau\in\oist$ we can use our oracle for $0'$ to decide
  whether $P\cap\llb\tau\rrb$ is empty or not.  Thus, if already know
  $T(\sigma)$ for some $\sigma\in\oist$, we can then compute
  $T(\sigma\cat\la i\ra)$ for $i\in\oi$, because $T(\sigma\cat\la
  i\ra)$ is the smallest $\tau\supseteq T(\sigma)\cat\la i\ra$ such
  that $P\cap\llb\tau\cat\la j\ra\rrb$ is nonempty for all $j\in\oi$.
  Thus $T\leT0'$, Q.E.D.
\end{proof}

In our proof of Theorem \ref{thm:Q}, the treemap $T=T_Q$ corresponding
to $Q$ will be \emph{uniform}, in the sense that $|T(\sigma)|$ will
depend only on $|\sigma|$.  This feature of $T$ will be for
convenience only, but $T$ will also have another key property, which
reads as follows.  For all $e,l,n\in\NN$ with $e\le l$ and $n\le2^l$
and all pairwise distinct $\sigma_1,\ldots,\sigma_n\in\oi^l$, either
$\llb T(\sigma_1)\rrb\times\cdots\times\llb T(\sigma_n)\rrb\cap
P_e=\emptyset$ or $(Q\cap\llb
T(\sigma_1)\rrb)\times\cdots\times(Q\cap\llb T(\sigma_n)\rrb)\subseteq
P_e$.  Thus by Lemma \ref{lem:treemap} the ``moreover'' clause of
Theorem \ref{thm:Q} will hold with $m=m_e=|T(\sigma)|$ for all
$\sigma\in\oi^e$.

\begin{proof}[Proof of Theorem \ref{thm:Q}]
  We shall construct $T=T_Q$ as the limit $T=\lim_sT_s$ of a recursive
  sequence of recursive uniform treemaps $T_s$, $s\in\NN$.  This will
  be a pointwise limit as $s\to\infty$, in the sense that for each
  $\sigma\in\oist$ we shall have $T(\sigma)=T_s(\sigma)$ for all
  sufficiently large $s$.  Also, the treemaps $T_s$ will be
  \emph{nested}, in the sense that for all $s\in\NN$ and all
  $\sigma\in\oist$ there will be a $\rho\in\oist$ such that
  $T_{s+1}(\sigma)=T_s(\rho)$.  From this it follows that
  $[T_{s+1}]\subseteq[T_s]$ for all $s$, and that
  $[T]=\bigcap_s[T_s]$.  Thus $Q=[T]$ will be a nonempty perfect
  $\Pi^0_1$ subclass of $\oiNN$.

  In presenting our construction of $T=\lim_sT_s$, we shall use the
  notation
  \begin{center}
    $P_{e,s}=\{X\in\oiNN\mid\varphi_{e,s}^{(1),X\res s}(0)\uparrow\}$.
  \end{center}
  Note that $P_{e,s}=\bigcup_\tau\llb\tau\rrb$ where the union is
  taken over a finite subset of $\oi^s$.  Note also that
  $P_{e,s+1}\subseteq P_{e,s}$ for all $s\in\NN$, and that
  \begin{center}
    $P_e=\{X\in\oiNN\mid\varphi_e^{(1),X}(0)\uparrow\}=\bigcap_sP_{e,s}$.
  \end{center}

  We now offer a preliminary account of the construction and proof.
  To each $e,l,n\in\NN$ with $e\le l$ and $n\le2^l$ and each pairwise
  distinct sequence $\sigma_1,\ldots,\sigma_n\in\oi^l$, we associate a
  \emph{requirement $R(e,\sigma_1,\ldots,\sigma_n)$ at level $l$}.
  Intuitively, the purpose of this requirement is to insure that
  $(\llb T(\sigma_1)\rrb\times\cdots\times\llb T(\sigma_n)\rrb)\cap
  P_e=\emptyset$ ``if possible.''  The strategy here will be to
  attempt to arrange that $(\llb
  T_s(\sigma_1)\rrb\times\cdots\times\llb T_s(\sigma_n)\rrb)\cap
  P_{e,s}=\emptyset$ for all sufficiently large $s$, ``if possible.''
  We shall argue that if this attempt fails, then $(Q\cap\llb
  T_s(\sigma_1)\rrb)\times\cdots\times(Q\cap\llb
  T_s(\sigma_n)\rrb)\subseteq P_{e,s}$ for all sufficiently large $s$,
  hence $(Q\cap\llb T(\sigma_1)\rrb)\times\cdots\times(Q\cap\llb
  T(\sigma_n)\rrb)\subseteq P_e$.

  We now give the detailed construction of $T_s$ for all $s\in\NN$.
  Begin by fixing a recursive linear ordering of all of the
  requirements, called the \emph{priority ordering}.  Arrange this
  ordering so that for each $l\in\NN$, all requirements at level $l$
  are of lower priority than all requirements at level $<l$.  The idea
  here is that requirements at level $l$ can be ``injured'' by
  requirements at level $<l$ but will never be ``injured'' by
  requirements at level $\ge l$.  Note that for each $l$ there are
  only finitely many requirements at level $\le l$.  We proceed by
  induction on $s$.

  Stage 0.  Let $T_0(\nu)=\nu$ for all $\nu\in\oist$.  Clearly $T_0$
  is a uniform treemap.

  Stage $s+1$.  Assume inductively that we have defined a uniform
  treemap $T_s$.  Our task at this stage is to define $T_{s+1}$.  A
  requirement $R(e,\sigma_1,\ldots,\sigma_n)$ at level $l$ is said to
  be \emph{requesting attention at stage $s$} if $(\llb
  T_s(\sigma_1)\rrb\times\cdots\times\llb T_s(\sigma_n)\rrb)\cap
  P_{e,s}\ne\emptyset$ but there exist $k\le s$ and
  $\rho_1,\ldots,\rho_n\in\oi^k$ such that $\rho_i\supset\sigma_i$ for
  all $i=1,\ldots,n$ and $(\llb T_s(\rho_1)\rrb\times\cdots\times\llb
  T_s(\rho_n)\rrb)\cap P_{e,s}=\emptyset$.  If no requirements are
  requesting attention, do nothing, i.e., let $T_{s+1}(\nu)=T_s(\nu)$
  for all $\nu\in\oist$.  Otherwise, let $R_s$ be the requirement of
  highest priority which is requesting attention.  For this
  requirement only, choose $k>l\ge e$ and $\rho_1,\ldots,\rho_n$ as
  above and define $T_{s+1}$ as follows.  For each $\nu\in\oi^{<l}$
  let $T_{s+1}(\nu)=T_s(\nu)$.  For each $i=1,\ldots,n$ let
  $T_{s+1}(\sigma_i)=T_s(\rho_i)$.  For each $\sigma\in\oi^l$ other
  than $\sigma_1,\ldots,\sigma_n$, let
  $T_{s+1}(\sigma)=T_s(\sigma\cat\la\underbrace{0,\ldots,0}_{k-l}\ra)$.
  For each $\sigma\in\oi^l$ and each $\nu\in\oist$, let
  $T_{s+1}(\sigma\cat\nu)=T_s(\rho\cat\nu)$ where $\rho$ is such that
  $T_{s+1}(\sigma)=T_s(\rho)$.  Clearly $T_{s+1}$ is a treemap, and it
  is uniform because $T_s$ is uniform and for each $\sigma\in\oi^l$ we
  have $T_{s+1}(\sigma)=T_s(\rho)$ for some $\rho\in\oi^k$.

  We now have a recursive nested sequence of uniform treemaps $T_s$.
  As $s$ goes to infinity, consider the history of a single
  requirement $R=R(e,\sigma_1,\ldots,\sigma_n)$ at level $l$.  Let us
  say that $R$ is \emph{satisfied at stage $s$} if $(\llb
  T_s(\sigma_1)\rrb\times\cdots\times\llb T_s(\sigma_n)\rrb)\cap
  P_{e,s}=\emptyset$, otherwise \emph{unsatisfied at stage $s$}.  By
  construction, if $R=R_s$ then $R$ is unsatisfied at stage $s$ but
  becomes satisfied at stage $s+1$.  And if $R$ is satisfied at stage
  $s$, it remains satisfied at stage $s+1$ unless $R_s$ is at level
  $<l$.  Let us say that $R$ is \emph{injured at stage $s$} if the
  latter case holds, i.e., if the level of $R_s$ is less than the
  level of $R$.  By induction along the priority ordering, we now see
  that there are only finitely many stages $s$ at which $R$ is
  injured, and only finitely many stages $s$ at which $R$ requests
  attention.  This holds for all of the finitely many requirements at
  level $\le l$, so for all $\sigma\in\oi^l$ and all sufficiently
  large $s$ we have $T_{s+1}(\sigma)=T_s(\sigma)$.  We now have a
  uniform treemap $T$ defined by $T(\sigma)=\lim_sT_s(\sigma)$, so
  letting $Q=[T]=\bigcap_s[T_s]$ we have a nonempty perfect $\Pioi$
  class $Q\subseteq\oiNN$.

  We claim that for each requirement $R$ at level $l$ as above, either
  $(\llb T(\sigma_1)\rrb\times\cdots\times\llb T(\sigma_n)\rrb)\cap
  P_e=\emptyset$ or $(Q\cap\llb
  T(\sigma_1)\rrb)\times\cdots\times(Q\cap\llb
  T(\sigma_n)\rrb)\subseteq P_e$.  Suppose not.  Then $(\llb
  T(\sigma_1)\rrb\times\cdots\times\llb T(\sigma_n)\rrb)\cap
  P_e\ne\emptyset$ and
  \begin{center}
    $((Q\cap\llb T(\sigma_1)\rrb)\times\cdots\times(Q\cap\llb
    T(\sigma_n)\rrb))\setminus P_e\ne\emptyset$.
  \end{center}
  For $i=1,\ldots,n$ fix $X_i\in Q\cap\llb T(\sigma_i)\rrb$ so that
  $X_1\oplus\cdots\oplus X_n\notin P_e$.  Let $m$ be so large that
  $(\llb X_1\res m\rrb\times\cdots\times\llb X_n\res m\rrb)\cap
  P_e=\emptyset$.  Let $k>l$ and $\rho_1,\ldots,\rho_n\in\oi^k$ be
  such that $T(\rho_i)\subset X_i$ and $|T(\rho_i)|\ge m$ for all
  $i=1,\ldots,n$.  Thus we have $\rho_i\supset\sigma_i$ for all
  $i=1,\ldots,n$, and $(\llb T(\rho_1)\rrb\times\cdots\times\llb
  T(\rho_n)\rrb)\cap P_e=\emptyset$.  It follows by compactness that
  for all sufficiently large $s$ we have $(\llb
  T(\rho_1)\rrb\times\cdots\times\llb T(\rho_n)\rrb)\cap
  P_{e,s}=\emptyset$.  And we can also let $s$ be so large that $e\le
  l<k\le s$ and $T_s(\sigma_i)=T(\sigma_i)$ and
  $T_s(\rho_i)=T(\rho_i)$ for all $i=1,\ldots,n$.  So now we see that
  $(\llb T_s(\sigma_1)\rrb\times\cdots\times\llb
  T_s(\sigma_n)\rrb)\cap P_{e,s}\ne\emptyset$ and $(\llb
  T_s(\rho_1)\rrb\times\cdots\times\llb T_s(\rho_n)\rrb)\cap
  P_{e,s}=\emptyset$ for all suffiently large $s$.  Thus our
  requirement $R$ is requesting attention at all sufficiently large
  stages $s$.  This contradiction proves our claim.

  The above claim gives us the principal conclusion of Theorem
  \ref{thm:Q}.  And then, as we have seen, Lemma \ref{lem:treemap}
  gives us the ``moreover'' clause with $m=m_e$.  The proof of Theorem
  \ref{thm:Q} is now complete.
\end{proof}


\end{document}